\documentclass[a4paper]{article}
\usepackage{amssymb,amsmath,amsthm,verbatim,amscd,enumerate,url}
\usepackage{graphicx}

\DeclareGraphicsExtensions{.png}
\DeclareMathOperator{\Supp}{Supp}

\DeclareMathOperator{\girth}{girth}
\DeclareMathOperator{\geod}{geod}

\title{Embeddable box spaces of free groups}
\author{A. Khukhro\\ 
\small
\textit{Universit\'{e} de Neuch\^{a}tel,} \   \texttt{a.khukhro@gmail.com}}

\parskip\bigskipamount

\theoremstyle{plain}

\newtheorem*{thm}{Theorem}

\newtheorem{Thm}{Theorem}
\newtheorem{Lemma}[Thm]{Lemma}
\newtheorem{Prop}[Thm]{Proposition}
\newtheorem{Corollary}[Thm]{Corollary}

\theoremstyle{definition}

\newtheorem*{definition}{Definition}

\newtheorem*{acknow}{Acknowledgements}

\begin{document}
\maketitle

\begin{abstract}
We generalize the construction of Arzhantseva, Guentner and Spakula of a box space of the free group which admits a coarse embedding into Hilbert space. We show that for a finitely generated free group, the box space corresponding to the derived $m$-series (for any integer $m\geq 2$) coarsely embeds into Hilbert space. This gives new examples of metric spaces with bounded geometry which coarsely embed into Hilbert space but do not have Yu's property A.
\end{abstract}


\section*{Introduction}
A \emph{coarse embedding} of one metric space into another generalizes the notion of a quasi-isometric embedding, by allowing the functions which control how the metric is distorted to be non-linear. 
\begin{definition}
Let $(X,d_X)$ and $(Y,d_Y)$ be metric spaces.
$X$ \emph{coarsely embeds} into $Y$ if there is a map $F:X\longrightarrow Y$ such that there exist non-decreasing functions $\rho_\pm: \mathbb{R}_+ \longrightarrow \mathbb{R}_+$ with $\lim_{t\rightarrow \infty}\rho_\pm(t)=\infty$ and
\begin{equation*}
\rho_-(d_X(x,x'))\leq d_Y(F(x),F(x'))\leq \rho_+(d_X(x,x'))
\end{equation*}
for all $x,x' \in X$. $X$ and $Y$ are \emph{coarsely equivalent} if there exists a coarse embedding $F:X\longrightarrow Y$ and $C>0$ such that for each $y\in Y$, $d_Y(y,F(X))<C$.
\end{definition}
For Cayley graphs of finitely generated groups (and, more generally, all quasi-geodesic spaces), if two such spaces are coarsely equivalent, then they are necessarily quasi-isometric.
 
We are interested in spaces which admit a coarse embedding into Hilbert space, since it implies the coarse Baum--Connes conjecture and, in the case that the space is a Cayley graph of a finitely generated group, the Novikov conjecture \cite{Yu}. In the same remarkable paper \cite{Yu}, Yu defines a geometric property which implies coarse embeddability into Hilbert space. We give an equivalent definition from \cite{Tu}.
\begin{definition}[\cite{Yu}, \cite{Tu}]
A discrete metric space $(X,d)$ with bounded geometry has \emph{property A} if and only if for every $R>0$ and $\varepsilon >0$ there exists an $S>0$ and a  function $\phi: X \longrightarrow \ell^2(X)$ such that $\|\phi(x)\|=1$ for all $x \in X$ and such that for all $x_1, x_2 \in X$:
\begin{description}
\item[$(1)$]
if $d(x_1,x_2)\leq R$ then $|1-\left\langle \phi(x_1), \phi(x_2)\right\rangle| \leq \varepsilon$, and
\item[$(2)$]
$\Supp\phi(x) \subset B_{S}(x)$ for all $x\in X$.
\end{description}
\end{definition}
One of the few examples of spaces without property A can be given in the form of \emph{box spaces}, which are spaces with bounded geometry constructed using quotients of residually finite groups.

Given a sequence of metric spaces $\{X_n,d_n\}$, their \emph{coarse disjoint union} is the space $\sqcup_n X_n$ with metric $d$ such that $d$ is $d_n$ when restricted to each component $X_n$, and the distance between two distinct components is chosen to be greater than the maximum of their diameters (note that any two such choices of metric result in coarsely equivalent spaces).
Let $G$ be a finitely generated residually finite group and let $\{N_i\}$ be a collection of finite index nested normal subgroups of $G$, for which the intersection $\cap_{i\in \mathbb{N}} N_i$ is trivial.
Note that for a finitely generated residually finite group $G$, there always exists a such a collection of subgroups $\{N_i\}$.

\begin{definition}
The \emph{box space} $\Box_{\{N_i\}} G$ of $G$ corresponding to $\{N_i\}$ is the coarse disjoint union $\sqcup_i G/N_i$ of finite quotient groups of $G$, where each quotient is endowed with the Cayley graph metric induced by the image of the generating set of $G$.
\end{definition}

For a residually finite group $G$, and $\{N_i\}$ a sequence of subgroups as above, we have the following links between analytic properties of $G$ and geometric properties of the box space:
\begin{align*}
G \text{ amenable } &\Longleftrightarrow \Box_{\{N_i\}}G \text{ property A,}\\
G \text{ Haagerup } &\Longleftarrow \Box_{\{N_i\}}G \text{ coarsely embeddable into Hilbert space,}\\
G \text{ property (T) } &\Longrightarrow \Box_{\{N_i\}}G \text{ expander.}
\end{align*}
The proofs of the first two statements can be found in \cite{Roe}, and the third in \cite{Mar}.
Note that the last two implications are not reversible. When property (T) above is replaced by a weaker property called property ($\tau$) \cite{LZ}, the implication becomes an equivalence. 
In \cite{WY}, Willett and Yu define a new geometric property called \emph{geometric property (T)}. A box space having this property is equivalent to the group having property (T). A recent result of Chen, Wang and Wang characterizes the Haagerup property for residually finite groups in terms of their box spaces admitting a \emph{fibred coarse embedding} into Hilbert space \cite{CWW}. 

The equivalence between amenability of a group and property A of its box space provides us with a source of examples of spaces without A.
Box spaces were also the first examples of spaces which coarsely embed into Hilbert space but do not have property A. Such a space with unbounded geometry was first exhibited by Nowak \cite{Now}, in the form of a disjoint union of cubes of increasing dimension. The first bounded geometry example was later given by Arzhantseva, Guentner and Spakula. 
\begin{definition}
Given $m\in \mathbb{N}$ and a group $G$, the \emph{derived $m$-series} of $G$ is a sequence of subgroups defined inductively by $G_1=G$, $G_{i+1}=[G_i,G_i]G_i^m$, where $G_i^m$ is the subgroup of $G$ generated by $m$th powers of elements of $G_i$.
\end{definition}
When $G$ is free, the intersection $\cap G_i$ of all the $G_i$ is trivial by a theorem of Levi (see Proposition 3.3 in Chapter 1 of \cite{LS}), since each $G_i$ is a proper characteristic subgroup of the previous $G_{i-1}$.
For free groups it thus makes sense to talk about the box space corresponding to the derived $m$-series, for $m\in \mathbb{N}$. 

\begin{thm}[\cite{AGS}]
Given a finitely generated free group, the box space corresponding to the derived $2$-series coarsely embeds into Hilbert space.
\end{thm}

We have the following stability result for the property of having a box space which coarsely embeds into Hilbert space.
\begin{thm}[\cite{Khu}]
Let $\Gamma$ be the semidirect product $H\rtimes G$ of two finitely generated residually finite groups $H$ and $G$ with a nested sequence of finite index characteristic subgroups $\{N_i\}$ of $H$ with $\cap N_i = 1$, such that $\Box_{\{N_i\}}H$ embeds coarsely into Hilbert space, and G is amenable. Then $\Gamma$ has a box space which coarsely embeds into Hilbert space.
\end{thm}
In particular, this applies to (finitely generated free)-by-cyclic groups. 
In this paper, we generalize the result of \cite{AGS} in a different direction.
\begin{thm}
Given a finitely generated free group and an integer $m\geq 2$, the box space corresponding to the $m$-derived series coarsely embeds into Hilbert space.
\end{thm}
The case $m=2$ is the main result of \cite{AGS}.
This is a first step towards classifying which box spaces of finitely generated free groups admit a coarse embedding into Hilbert space. Note that there exist box spaces of free groups which are expanders, and hence do not embed into Hilbert space.

From now on, we will say that a metric space is \emph{embeddable} if it embeds coarsely into Hilbert space, and $m$ will always be an integer $\geq 3$. 

\begin{acknow}
The author thanks Alain Valette for enjoyable discussions, and for his very helpful comments on a draft of this paper. The author also wishes to thank Antoine Gournay for simplifying part of the proof of Proposition \ref{Compare}, which was previously combinatorial in nature. 
\end{acknow}

\section*{Outline}
Let $G_i$ be the derived 2-series of a finitely generated free group $G=\mathbb{F}_n$.
In \cite{AGS}, the authors use the fact that each $G/G_i$ is the $\mathbb{Z}_2$-homology cover of the previous $G/G_{i-1}$ to induce a wall structure on each $G/G_i$. This wall structure gives rise to a wall metric on the box space $\Box_{\{G_i\}}G$ corresponding to the derived 2-series, with respect to which the box space embeds into Hilbert space. The final step is showing that the wall metric and the original metric are coarsely equivalent, using the fact that the girth of the graphs tends to infinity. 

The main idea of our generalization is that instead of using a wall space structure coming from the covering groups being cubes $\oplus^n \mathbb{Z}_2$, one can obtain a metric with respect to which the box space embeds using the result of Nowak that given any finite group $F$, the disjoint union $\sqcup_{n\in \mathbb{N}} \oplus^n F$ coarsely embeds into Hilbert space. This result appears as part of Nowak's example of a non-bounded geometry metric space which embeds into Hilbert space but does not have property A \cite{Now}. 

We consider the case $F=\mathbb{Z}_m$ for $m\geq 3$. Given a sequence of finite $2$-connected graphs $\{X_n\}$ with girth tending to infinity (and an additional technical assumption about the symmetry of the graphs), we look at the coarse disjoint union $\sqcup_n \widetilde{X}_n$ of their covers corresponding to the quotients $$\pi_1(X_n)\longrightarrow \pi_1(X_n)/[\pi_1(X_n),\pi_1(X_n)]\pi_1(X_n)^m,$$ which we call the $\mathbb{Z}_m$-\emph{homology covers}, with their natural graph metrics. We use the $\oplus^n \mathbb{Z}_m$ structure induced in each cover to define a new metric on this coarse disjoint union. Thanks to the result of Nowak, one can show that $\sqcup_n \widetilde{X}_n$ coarsely embeds into Hilbert space with respect to this new metric. We then use large girth to show that the two metrics on $\sqcup_n \widetilde{X}_n$ are coarsely equivalent. 

To find new examples of embeddable box spaces, we look at the derived $m$-series $\{G_i\}$ of a finitely generated free group $G$ for $m\in \mathbb{N}$. Each successive quotient in $\Box_{\{G_i\}}G$ is the $\mathbb{Z}_m$-homology cover of the previous quotient. Being a box space of the free group, we have $\girth(G/G_i)\rightarrow \infty$ and so we can apply the general results described above to show that $\Box_{\{G_i\}}G$ is embeddable.

\section*{Covers and metrics}

Let us first describe the general construction of the cover $\widetilde{X}$ of a finite graph $X$ corresponding to a finite quotient $K$ of $\pi_1(X)$. Throughout, we will assume that $X$ is $2$-connected, i.e. removing any edge leaves $X$ connected. 
Let $\rho$ be the surjective homomorphism $\rho: \pi_1(X) \twoheadrightarrow K$.

Denote the vertex set of  $X$ by $V(X)$ and the edge set by $E(X)$. Choose a maximal tree $T\subset X$. The set of edges $\{e_1, e_2, ..., e_r\}$ which are not in the maximal tree $T$ correspond to free generators of $\pi_1(X)$, and so we can consider their image under the quotient map $\rho$. The cover of $X$ corresponding to $\rho$ is the finite graph $\widetilde{X}$ with vertex set given by $$V(\widetilde{X})= V(X)\times K$$ and edge set given by $$E(\widetilde{X})= E(X)\times K.$$ We now just need to specify the vertices which are connected by each edge in $E(\widetilde{X})$. 

Given an edge $(e,k)\in E(\widetilde{X})$ (where $e\in E(X)$ and $k\in K$), let $v$ and $w$ be the vertices of $X$ connected by $e$. There are two cases: $e\in T$ and $e\notin T$. If $e\in T$, let $(e,k)$ connect the vertices $(v,k)$ and $(w,k)$. If $e\notin T$, let $(e,k)$ connect $(v,k)$ and $(w,\rho(e)k)$. The graph $\widetilde{X}$ defined in this way is the cover of $X$ corresponding to $\rho: \pi_1(X) \twoheadrightarrow K$. Note that the cover we obtain does not depend on the choice of spanning tree or on the chosen orientation of edges, i.e. it is unique up to graph isomorphism commuting with the covering projections (see Proposition 2.2 of \cite{AGS}). 

The covering map $\pi: \widetilde{X}\rightarrow X$ is given by $(e,k)\mapsto e$ and $(v,k)\mapsto v$. We can consider the subgraphs $V(X)\times k$ as $k$ ranges over the elements of $K$. Following \cite{AGS}, we will call these subgraphs \emph{clouds}. Note that collapsing the clouds to points yields the Cayley graph of the group $K$ with respect to the generating set consisting of the images of the free generating set of $\pi_1(X)$.

\includegraphics[scale=0.53]{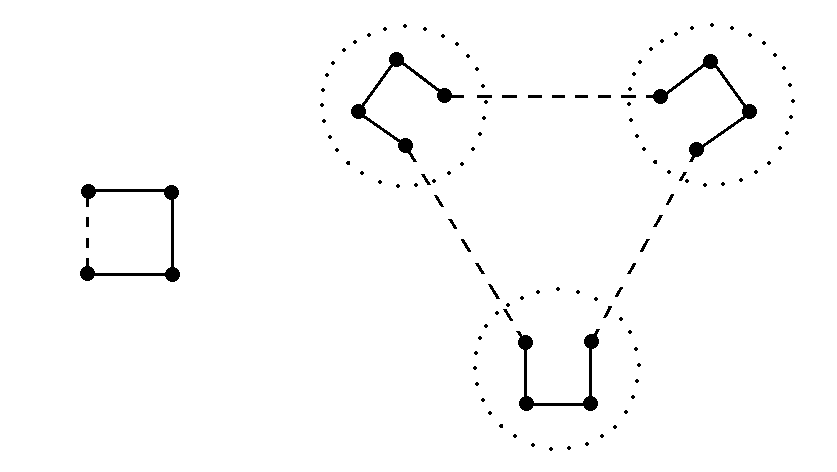}

$\ \ \ \ \ \ \ \ \ \ \ \ \ \ \ X \ \ \ \ \ \ \ \ \ \ \ \ \ \ \ \ \ \ \ \ \ \ \ \ \ \ \ \ \ \ \ \ \ \ \ \ \ \ \ \ \ \ \ \ \ \ \ \ \widetilde{X}$

In the example above, $\widetilde{X}$ is the cover of $X$ corresponding to the quotient $\rho: \pi_1(X)\cong \mathbb{Z} \twoheadrightarrow \mathbb{Z}_3$. We see copies of the (solid line) maximal tree of $X$ inside the three clouds corresponding to elements of $\mathbb{Z}_3$, with edges which are lifts of the edge not in the maximal tree of $X$ (represented by broken lines) connecting the clouds according to the quotient map $\rho$. 

We will concentrate on the case where the cover $\widetilde{X}$ corresponds to the quotient $$\pi_1(X)\longrightarrow \pi_1(X)/[\pi_1(X),\pi_1(X)]\pi_1(X)^m\cong \oplus^r \mathbb{Z}_m,$$ which we call the $\mathbb{Z}_m$-\emph{homology cover} of $X$. Note that the Cayley graph of $ \oplus^r \mathbb{Z}_m$ (where $r$ is the free rank of $\pi_1(X)$) with respect to the image of the free generating set of $\pi_1(X)$ is the same as taking the natural generating set for $\oplus^r \mathbb{Z}_m$, namely, one element from each copy of $\mathbb{Z}_m$. We will refer to the corresponding word metric as $d_T$. 

For each element $x\in \widetilde{X}$, denote by $C^T_x$ the cloud (with respect to the maximal spanning tree $T$) containing $x$.
Since collapsing the clouds of $\widetilde{X}$ to points gives us the space $(\oplus^r \mathbb{Z}_m, d_T)$, the clouds are in one-to-one correspondence with elements of $\oplus^r \mathbb{Z}_m$. We will refer to clouds and points in $\oplus^r \mathbb{Z}_m$ interchangeably.

We will now introduce two metrics on $\widetilde{X}$. The first, $d$, is the natural graph metric on $\widetilde{X}$. The second, $d_Q$, is a metric which we will see comes from the $\oplus^r \mathbb{Z}_m$ structure present in $\widetilde{X}$. For each $x, y \in \widetilde{X}$, choose a geodesic $\geod[x,y]$ between $x$ and $y$ with respect to the metric $d$. 

For each edge $e\in E(X)$, choose an orientation. 
Define a function $$\phi: E(X)\times \widetilde{X}\times\widetilde{X}\longrightarrow \mathbb{N}$$ by setting $\phi(e,x,y)$ to be the smallest non-negative residue modulo
$m$ of $$||\pi^{-1}(e)\cap \geod[x,y]|-|\pi^{-1}(e^{-1})\cap \geod[x,y]||,$$ where by $|\pi^{-1}(e)\cap \geod[x,y]|$, we mean the number of times that a lift of $e$ occurs in the geodesic (with positive orientation), and by $|\pi^{-1}(e^{-1})\cap \geod[x,y]|$, we mean the number of times a lift of the edge $e$ occurs with reversed orientation. Define $d_Q$ as follows:
\begin{equation*}
d_Q(x,y):= \sum_{e\in E(X)} \sum_{T: e\notin T} \frac{1}{N_e} \min\{\phi(e,x,y), m-\phi(e,x,y)\}.
\end{equation*}
Here, the first sum ranges over all edges in $X$, the second sum ranges over all maximal spanning trees $T$ of $X$ which do not contain a given edge $e$, and $N_e$ denotes the number of maximal spanning trees which do not contain a given edge $e$. Since $X$ is always assumed to be $2$-connected, the second sum is never empty and $N_e$ is non-zero for all $e$.

Suppose that $X$ is such that $N_e=N$ is independent of the choice of edge $e$ (which we will show to be the case in the situation we will be interested in). Then we have
\begin{equation*}
d_Q(x,y)=\frac{1}{N} \sum_{T} \sum_{e\notin T} \min\{\phi(e,x,y), m-\phi(e,x,y)\}.
\end{equation*}
Now note that given a maximal spanning tree $T$ and $x, y \in \widetilde{X}$, the sum $$\sum_{e\notin T} \min\{\phi(e,x,y), m-\phi(e,x,y)\}$$ is exactly the distance between the clouds $C^T_x$ and $C^T_y$ containing $x$ and $y$ respectively in $\oplus^r \mathbb{Z}_m$ with the metric $d_T$. 
This is because given an element $(z_1,...,z_r)$ in $\oplus^r \mathbb{Z}_m$, written additively, the geodesics from the identity to $(z_1,...,z_r)$ are exactly those paths which for each $i$, contain $z_i$ edges corresponding to the generator of the $i$th factor if $z_i\leq m/2$, or $m-z_i$ edges corresponding to the inverse of the generator of the $i$th factor if $z_i> m/2$, in any order. The minimum in the sum above therefore ensures that we get the $d_T$ geodesic. Thus, we have
$$d_Q(x,y)=\frac{1}{N} \sum_{T} d_T(C^T_x,C^T_y).$$
Note that this in particular proves that $d_Q$ is a pseudometric, since the triangle inequality is obvious from the above. We will see that $d(x,y)=0$ if and only if $x=y$ later: it will follow from Proposition \ref{Compare}, when we compare $d_Q$ with the metric $d$.

Note also that one has 
\begin{align*}
d_Q(x,y)&=\frac{1}{N} \sum_{T} \sum_{e\notin T} \min\{\phi(e,x,y), m-\phi(e,x,y)\}\\
&=\sum_{e\in E(X)} \sum_{T: e \notin T} \frac{1}{N} \min\{\phi(e,x,y), m-\phi(e,x,y)\}\\
&=\sum_{e\in E(X)} \min\{\phi(e,x,y), m-\phi(e,x,y)\}\\
&\leq \sum_{e\in E(X)}\phi(e,x,y)\\
&\leq\sum_{e\in E(X)}||\pi^{-1}(e)\cap \geod[x,y]|-|\pi^{-1}(e^{-1})\cap \geod[x,y]||\\
&\leq\sum_{e\in E(X)} |\pi^{-1}(e)\cap \geod[x,y]| + |\pi^{-1}(e^{-1})\cap \geod[x,y]|\\
&=|\geod[x,y]|= d(x,y).
\end{align*}

For us, $d_Q$ will be the metric which we will use to embed $\widetilde{X}$ into Hilbert space. However, we are interested in embedding $\widetilde{X}$ with its original metric, so we now need a way to compare the two metrics. The following proposition, which is inspired by Proposition 3.11 in \cite{AGS}, will do this for us. 

\begin{Prop}\label{Compare}
If $X$ and $\widetilde{X}$ are as described above, then for every $x, y \in \widetilde{X}$, we have
$$d_Q(x,y)<\girth(X) \iff d(x,y)<\girth(X)$$
and if the above inequalities hold, then $d_Q(x,y)=d(x,y)$.
\end{Prop}

\begin{proof}
First assume that $d(x,y)<\girth(X)$. We have seen that $d_Q\leq d$ and so we have that $d_Q(x,y)< \girth(X)$. We now prove that when $d(x,y)<\girth(X)$, we have $d_Q(x,y)=d(x,y)$.

Projecting a geodesic $\geod[x,y]$ for the metric $d$ onto $X$, we see that there can be no repeated edges in this path since $d(x,y)<\girth(X)$. We now make the remark that in $\oplus^r \mathbb{Z}_m$, any path without multiple occurrences of edges labelled by the same generator must be a geodesic. Hence, for any choice of maximal tree $T$, the path between $C^T_x$ and $C^T_y$ induced in the corresponding $\oplus^r \mathbb{Z}_m$ by the $d$-geodesic must be a $d_T$-geodesic. We thus have (using the fact that a lift of either $e$ or $e^{-1}$ appears in the $d$-geodesic at most once):
\begin{align*}
d_Q(x,y)&=\frac{1}{N} \sum_{T} d_T(C^T_x,C^T_y)\\
&=\frac{1}{N} \sum_{T} \sum_{e \notin T} |\pi^{-1}(e)\cap \geod[x,y]|+|\pi^{-1}(e^{-1})\cap \geod[x,y]|\\
&=\sum_{e\in E(X)} \sum_{T:e \notin T} \frac{1}{N} |\pi^{-1}(e)\cap \geod[x,y]|+|\pi^{-1}(e^{-1})\cap \geod[x,y]|\\
&=\sum_{e\in E(X)} |\pi^{-1}(e)\cap \geod[x,y]|+|\pi^{-1}(e^{-1})\cap \geod[x,y]|\\
&=|\geod[x,y]|=d(x,y).
\end{align*}

Now it remains to show that $d_Q(x,y)<\girth(X)$ implies $d(x,y)<\girth(X)$. Assume that $d_Q(x,y)<\girth(X)$ and consider the projection $p[\pi(x),\pi(y)]$ of a $d$-geodesic onto $X$. Note that this is a shortest path in $X$ which lifts to a $d$-geodesic in $\widetilde{X}$ (i.e. no shorter path in $X$ can be lifted to a path in $\widetilde{X}$ between $x$ and $y$), and that this path does not contain backtracks.

If there are no repeated edges in $p[\pi(x),\pi(y)]$, the path between $C^T_x$ and $C^T_y$ in $\oplus^r \mathbb{Z}_m$ induced by the $d$-geodesic is necessarily a geodesic. Hence if $d(x,y)\geq \girth(X)$, we have
\begin{align*}
d_Q(x,y)&=\frac{1}{N} \sum_{T} d_T(C^T_x,C^T_y)\\
&=\frac{1}{N} \sum_{T} \sum_{e \notin T} |\pi^{-1}(e)\cap \geod[x,y]|+|\pi^{-1}(e^{-1})\cap \geod[x,y]|\\
&=\sum_{e\in E(X)} \sum_{T:e \notin T} \frac{1}{N} |\pi^{-1}(e)\cap \geod[x,y]|+|\pi^{-1}(e^{-1})\cap \geod[x,y]|\\
&=\sum_{e\in E(X)} |\pi^{-1}(e)\cap \geod[x,y]|+|\pi^{-1}(e^{-1})\cap \geod[x,y]|\\
&=|\geod[x,y]|=d(x,y)\\
&\geq \girth(X),
\end{align*} 
and so $d(x,y)\geq \girth(X)$ would imply that $d_Q(x,y)\geq \girth(X)$ which is a contradiction. 

Now let us consider what happens when the path $p[\pi(x),\pi(y)]$ contains repeated edges.
The arguments used are of a \emph{modular graph theory} flavour. 
For simplicity, we will use the following definitions and prove a simple result before continuing with the proof of Proposition \ref{Compare}. Recall that $m$ is always assumed to be $\geq 3$.

\begin{definition}
Given a path $p[a,b]$ from $a\in V(X)$ to $b\in V(X)$ in a graph $X$, write $|e\cap p[a,b]|$ for the number of times $p[a,b]$ traverses $e$ in the positive direction and $|e^{-1}\cap p[a,b]|$ for the number of times $p[a,b]$ traverses $e$ in the opposite direction.
We will call an edge $e$ on this path $m$-\emph{repeated} if 
$$0\neq |e\cap p[a,b]|-|e^{-1}\cap p[a,b]|\equiv 0\mod{m}.$$
We say that two paths $p_1[a,b]$ and $p_2[a,b]$ from $a$ to $b$ are $m$-\emph{congruent} if for all $e\in E(X)$,
$$|e\cap p_1[a,b]|-|e^{-1}\cap p_1[a,b]|\equiv |e\cap p_2[a,b]|-|e^{-1}\cap p_2[a,b]| \mod{m}.$$
A shorter path which is $m$-congruent to a path $p[a,b]$ will be called an $m$-\emph{shortcut} for $p[a,b]$. 
\end{definition}

\begin{Lemma}\label{CongLifts}
Given two paths $p_1[a,b]$ and $p_2[a,b]$ from $a$ to $b$ in $X$ which are $m$-congruent, take the lifts of these paths in the $\mathbb{Z}_m$-homology cover $\widetilde{X}$ of $X$, both starting at a point $x\in \pi^{-1}(a)$. Then both of these lifts end at the same point $y\in \pi^{-1}(b)$ of $\widetilde{X}$.
\end{Lemma}

\begin{proof}
Pick some maximal spanning tree $T$ in $X$, and consider the clouds in $\widetilde{X}$ corresponding to $T$. Taking lifts of both of the paths starting at $x\in \pi^{-1}(a)$, the condition $$|e\cap p_1[a,b]|-|e^{-1}\cap p_1[a,b]|\equiv |e\cap p_2[a,b]|-|e^{-1}\cap p_2[a,b]| \mod{m}$$ 
for the edges not in $T$ implies that both of the lifts terminate in the same cloud, since the cloud depends only on the number of times the edges not in $T$ are traversed modulo $m$. This now uniquely determines the point $y\in \pi^{-1}(b)$ and so we are done.
\end{proof}

We now continue with the proof of Proposition \ref{Compare}. Recall that we are assuming that $d_Q(x,y)< \girth(X)$, and aiming to show that this implies $d(x,y)< \girth(X)$. We have proved this in the case that the path $p[\pi(x),\pi(y)]$ has no repeated edges. Observe that repeated edges in $p[\pi(x),\pi(y)]$ necessarily imply that $d(x,y)\geq \girth(X)$ and so it remains to prove that the assumptions of repeated edges in $p[\pi(x),\pi(y)]$ and $d_Q(x,y)< \girth(X)$ lead to a contradiction.

Consider the sum 
$$ \sum_{e\in E(X)} |e\cap p[\pi(x),\pi(y)]|-|e^{-1}\cap p[\pi(x),\pi(y)]| \delta_e \in \mathbb{Z}_m E(X).$$
Note that if we remove all $m$-repeated edges from the path $p[\pi(x),\pi(y)]$, we still obtain the same element of $\mathbb{Z}_m E(X)$. 
For each edge $e$, let $v_e$ denote the initial vertex and let $w_e$ denote the terminal vertex.
We can now apply the \emph{boundary operator modulo} $m$, i.e. $\partial_m: \mathbb{Z}_m E(X) \longrightarrow \mathbb{Z}_m V(X)$ given by 
$$\partial_m(\sum_{e\in E(X)} \alpha_e \delta_e)= \sum_{e\in E(X)}\alpha_e (\delta_{w_e}-\delta_{v_e}) = \sum_{v\in V(X)} \beta_v \delta_v.$$
Since we started with a path $p[\pi(x),\pi(y)]$, we have that 
$$\partial_m( \sum_{e\in E(X)} |e\cap p[\pi(x),\pi(y)]|-|e^{-1}\cap p[\pi(x),\pi(y)]| \delta_e)=\delta_{\pi(y)}-\delta_{\pi(x)}.$$

Remove the $m$-repeated edges from $p[\pi(x),\pi(y)]$. This yields the same element of $\mathbb{Z}_m E(X)$, and so when we apply the boundary operator $\partial_m$, we still get $\delta_{\pi(y)}-\delta_{\pi(x)}$. 

Assume first that $\pi(x)\neq \pi(y)$.
Considerations of the boundary operator above imply that $p[\pi(x),\pi(y)]$ with the $m$-repeated edges removed is the union of a path $p'[\pi(x),\pi(y)]$ from $\pi(x)$ to $\pi(y)$ and possibly some loops which are disjoint from this path.

If the graph of the path $p[\pi(x),\pi(y)]$ with the $m$-repeated edges removed is connected, i.e. there are no loops disjoint from $p'[\pi(x),\pi(y)]$, then $p[\pi(x),\pi(y)]$ once the $m$-repeated edges have been removed is simply the path $p'[\pi(x),\pi(y)]$, which is $m$-congruent to the path $p[\pi(x),\pi(y)]$. 

If $p'[\pi(x),\pi(y)]$ does not contain loops, it is also an $m$-shortcut for $p[\pi(x),\pi(y)]$, since we obtained it by removing $m$-repeated edges. 
Now Lemma \ref{CongLifts} tells us that two $m$-congruent paths can be lifted to the same path in $\widetilde{X}$, and so the existence of an $m$-shortcut for the path $p[\pi(x),\pi(y)]$ is a contradiction to this path being the shortest path in $X$ which can be lifted to a path in $\widetilde{X}$ between $x$ and $y$. 
 
If $p'[\pi(x),\pi(y)]$ contains loops, then each of the edges on such a loop must contribute at least 1 to 
$$\sum_{T: a \notin T}\frac{1}{N}\min\{\phi(a,x,y), m-\phi(a,x,y)\}$$
in the total sum
$$d_Q(x,y)= \sum_{a\in E(X)}\sum_{T: a \notin T}\frac{1}{N}\min\{\phi(a,x,y), m-\phi(a,x,y)\},$$
whence $d_Q(x,y)$ is at least the length of this loop, which contradicts our assumption that $d_Q(x,y)< \girth(X)$.

If the graph of the path $p[\pi(x),\pi(y)]$ has become disconnected upon removal of the $m$-repeated edges, this means that it is the disjoint union of a path and a non-zero number of loops. Thus, in our original path $p[\pi(x),\pi(y)]$, we have found at least one loop of edges which are traversed non-zero modulo $m$ times and so the same conclusion holds, as above.

Assume now that $\pi(x)= \pi(y)$. In this case, the graph of the path $p[\pi(x),\pi(y)]$ with the $m$-repeated edges removed must be trivial, or must be a union of a non-zero number of loops. If it is trivial, then the trivial path consisting of the vertex $\pi(x)$ is an $m$-shortcut for the path $p[\pi(x),\pi(y)]$ which is a contradiction, as above. If it is a union of a non-zero number of loops, we again deduce the contradiction $d_Q(x,y)\geq \girth(X)$.

This completes the proof of the implication $$d_Q(x,y)<\girth(X) \Rightarrow d(x,y)<\girth(X)$$ and so Proposition \ref{Compare} is proved. 

\end{proof}

We will rely on a result of Nowak \cite{Now} which tells us that the space $\sqcup_{i\in \mathbb{N}} \oplus^i \mathbb{Z}_m$ with the metric coming from the standard generating sets of each $\oplus^i \mathbb{Z}_m$ coarsely embeds into Hilbert space. More precisely, given a finite group $F$ with a fixed generating set $S$, metrize the disjoint union $\sqcup_{i\in \mathbb{N}} \oplus^i F$ by setting the metric on each $\oplus^i F$ to be the standard metric induced by $S$, and specifying that the distance between the $\oplus^i F$ must tend to infinity.

\begin{thm}[\cite{Now}, Theorem 5.1 (2)]
Given any finite group $F$, the (locally finite) metric space $\sqcup_{i\in \mathbb{N}} \oplus^i F$ admits a bi-Lipschitz embedding into $\ell^1$. 
\end{thm}

\begin{Prop}\label{IndepEmb}
Let $X$ be a finite $2$-connected graph such that the number $N_e=N$ of maximal spanning trees not containing a given edge $e$ is independent of the edge chosen. 
Its $\mathbb{Z}_m$-homology cover $\widetilde{X}$ admits a coarse embedding into Hilbert space with respect to the metric $d_Q$ such that the functions $\rho_{\pm}$ depend only on $m$, and not on $X$. 
\end{Prop}

\begin{proof}
Let $d_m$ be the metric on $\sqcup_{i\in \mathbb{N}} \oplus^i \mathbb{Z}_m$ which on each component is the Cayley graph metric coming from taking one generator for each factor of $\oplus^i \mathbb{Z}_m$, and such that distances between components tend to infinity. 
Nowak's theorem tells us that there is a bi-Lipschitz embedding $$\phi:\sqcup_{i\in \mathbb{N}} \oplus^i \mathbb{Z}_m \longrightarrow \ell^1,$$ i.e. there exists a $C>0$ such that
\begin{equation*}
\frac{1}{C} d_m(a,b)\leq \|\phi(a)-\phi(b)\|_1\leq C d_m(a,b)
\end{equation*}
for all $a,b\in \sqcup_{i\in \mathbb{N}} \oplus^i \mathbb{Z}_m$. 
Let $r$ denote the $m$-rank of $\pi_1(X)$.
Recall that for a point $x$ in the cover $\widetilde{X}$, $C^T_x$ denotes the cloud (corresponding to some point in $\oplus^{r} \mathbb{Z}_m$) containing $x$ with respect to the maximal tree $T$. 
Define $\psi:\widetilde{X} \longrightarrow \ell^1$ by $$x \longmapsto \frac{1}{N} \oplus_{T} \phi(C^T_x).$$
This embedding satisfies
$$\|\psi(x)-\psi(y)\|_1=\|\frac{1}{N}\oplus_{T} \phi(C^T_x)-\frac{1}{N}\oplus_{T} \phi(C^T_y)\|_1=\frac{1}{N}\sum_T \|\phi(C^T_x)-\phi(C^T_y)\|_1$$
and thus we have 
\begin{align*}
\frac{1}{C} d_Q(x,y)&=\frac{1}{C} \frac{1}{N}\sum_T d_T(C^T_x,C^T_y)\\
&\leq \frac{1}{N}\sum_T \|\phi(C^T_x)-\phi(C^T_y)\|_1\\
&=\|\psi(x)-\psi(y)\|_1\\
&\leq \frac{1}{N} C\sum_T d_T(C^T_x,C^T_y)\\
&=C d_Q(x,y).
\end{align*}
Note that $C$ only depends on $m$. Since $\ell^1$ coarsely embeds into $\ell^2$, the proof is complete.

\end{proof}

\section*{Box spaces}

We can now add all of the ingredients of the previous section to get the following general result, which will in particular apply to box spaces of free groups.

\begin{Thm}\label{Main}
Let $\{X_n\}$ be a sequence of $2$-connected finite graphs such that for each $n$, the number of maximal spanning trees in $X_n$ not containing a given edge does not depend on the edge. Given $m\in \mathbb{N}$, let $\{\widetilde{X}_n\}$ be the sequence of $\mathbb{Z}_m$-homology covers of the $X_n$. 
If $\girth(X_n)\rightarrow \infty$ as $n\rightarrow \infty$, then the coarse disjoint union $\sqcup_n \widetilde{X}_n$ coarsely embeds into Hilbert space. 
\end{Thm}

Let $N_n$ be the number of maximal spanning trees of $X_n$ not containing a given edge. 
Using the method of the previous section, one can define a metric $d_{Q_n}$ on each $\widetilde{X}_n$ by $d_{Q_n}(x,y)=\frac{1}{N_n} \sum_{T} d_T(C^T_x,C^T_y)$, where the sum is taken over all maximal spanning trees of $X_n$.
We will write $d_Q$ to mean the coarse disjoint union metric which is $d_{Q_n}$ on each component $\widetilde{X}_n$. Let $d$ denote the coarse disjoint union metric on $\sqcup_n \widetilde{X}_n$ which restricts to the natural graph metric on each component. 

We will first need the following proposition, which is proved exactly as Proposition 4.5 of \cite{AGS}, using the comparison of metrics on the scale of the girth that we proved in the previous section. 

\begin{Prop}\label{CE}
The identity map between the metric spaces formed by taking $\sqcup_n \widetilde{X}_n$ with the two different metrics $d$ and $d_Q$ is a coarse equivalence, i.e. the identity map and its inverse are both coarse embeddings. 
\end{Prop}

\begin{proof}
It was proved in the previous section that $d_{Q_n}$ is always smaller than the natural graph metric on each $\widetilde{X}_n$ and so we need only prove that for each $R>0$ there is an $S>0$ such that $d_Q(x,y)\leq R$ implies $d(x,y)\leq S$ for all $x,y \in \sqcup_n \widetilde{X}_n$. 

Given $R>0$, take $M>0$ such that for all $n,m\geq M$, we have $\girth(X_n)> R$ and the distance $d_Q$ between components $\widetilde{X}_n$ and $\widetilde{X}_m$ is greater than $R$. Let $S:=\max\{R,d(x,y):x,y \in \sqcup_{n<M} \widetilde{X}_n\}$. 

Now if $x,y\in \sqcup_n \widetilde{X}_n$ with $d_Q(x,y)\leq R$, then either $x,y \in \sqcup_{n<M} \widetilde{X}_n$ whence $d(x,y)\leq S$ by the definition of $S$, or $x,y \in \widetilde{X}_n$ for $n\geq M$. For this $\widetilde{X}_n$, Proposition \ref{Compare} tells us that the restriction of the identity map to balls of radius $R<\girth(X_n)$ in $\widetilde{X}_n$ is an isometry. This means we have $d(x,y)=d_Q(x,y)\leq R\leq S$ and so the proof is complete. 
\end{proof}

\begin{proof}[Proof of Theorem \ref{Main}]
Proposition \ref{IndepEmb} in the previous section tells us that each $\widetilde{X}_n$, being a $\mathbb{Z}_m$-homology cover, admits a coarse embedding into Hilbert space with respect to the metric $d_Q$. Moreover, the embedding functions $\rho_{\pm}$ depend only on $m\in \mathbb{N}$ and not on $X_n$, and so the embedding is uniform over all $n$.
Hence, $\sqcup_n \widetilde{X}_n$ with the metric $d_Q$ coarsely embeds into Hilbert space. Now the theorem is proved since by Proposition \ref{CE} above the metric $d_Q$ is coarsely equivalent to $d$, the metric arising from the natural graph metrics on $\sqcup_n \widetilde{X}_n$.
\end{proof}

We now apply Theorem \ref{Main} to certain box spaces of free groups. Let $G$ be the free group on $n$ generators, and let $m\in \mathbb{N}$. Let $\{G_i\}$ denote the derived $m$-series of $G$. Note that this is a nested sequence of finite index characteristic subgroups of $G$, so it makes sense to talk about the box space $\Box_{\{G_i\}}G$. Each successive quotient $G/G_{i+1}$ is the $\mathbb{Z}_m$-homology cover of $G/G_{i}$. 

To apply the above, set $X_n$ to be $G/G_n$. We then have that $\sqcup_n \widetilde{X}_n$ is equal to $\sqcup_n G/G_{n+1}$ and so to show that the box space $\Box_{\{G_i\}}G$ is embeddable, we need to show that the assumptions of Theorem \ref{Main} are satisfied. It is clear that the graphs $G/G_n$ are 2-connected, and since $\sqcup_n G/G_{n+1}$ is a box space of the free group, we have that $\girth(G/G_n)\rightarrow \infty$ as $n\rightarrow \infty$. It remains to show that for each $n$, the number of maximal spanning trees in $G/G_n$ not containing a given edge does not depend on the edge chosen. This is clear, since permuting the generators of $G$ induces an isomorphism of $G/G_n$, and hence a graph isomorphism of the corresponding Cayley graph.

We have obtained the following.

\begin{Corollary}
Given a finitely generated free group $\mathbb{F}_n$ and $m\in \mathbb{N}$ with $m\geq 3$, the box space corresponding to the $m$-derived series of $\mathbb{F}_n$ coarsely embeds into Hilbert space.
\end{Corollary}

In particular, taking such a box space of a free group $\mathbb{F}_n$ with $n\geq 2$ gives new examples of bounded geometry metric spaces which coarsely embed into Hilbert space but do not have property A, since non-abelian free groups are not amenable.


\begin{thebibliography}{[CWW]}
\bibitem[AGS]{AGS}
G. Arzhantseva, E. Guentner and J. Spakula, \emph{Coarse non-amenability and coarse embeddings}, Geom. Funct. Anal. 22, 2012
\bibitem[CWW]{CWW}
X. Chen, Q. Wang and X. Wang, \emph{Characterization of the Haagerup property by fibred coarse embedding into Hilbert space}, preprint, 2012
\bibitem[Khu]{Khu}
A. Khukhro, \emph{Box spaces, group extensions and coarse embeddings into Hilbert space}, J. Funct. Anal. 263, 2012
\bibitem[LZ]{LZ}
A. Lubotzky and A. \.{Z}uk, \emph{On property $(\tau)$}, preliminary version available online at \url{www.ma.huji.ac.il/~alexlub/}, 2003
\bibitem[LS]{LS}
R. Lyndon and P. Schupp, \emph{Combinatorial group theory}, Ergebnisse der Mathematik und ihrer Grenzgebiete 89, 1977
\bibitem[Mar]{Mar}
G. Margulis, \emph{Explicit group-theoretic constructions of combinatorial schemes and their applications in the construction of expanders and concentrators}, Problems Inform. Transmission 24, 1988
\bibitem[Now]{Now}
P. Nowak, \emph{Coarsely embeddable metric spaces without property A}, J. Funct. Anal. 252, 2007
\bibitem[Roe]{Roe}
J. Roe, \emph{Lectures on coarse geometry}, AMS University Lecture Series, Volume 31, 2003
\bibitem[Tu]{Tu}
J.-L. Tu, \emph{Remarks on Yu's property A for discrete metric spaces and groups}, Bull. Soc. Math. France 129, 2001
\bibitem[WY]{WY}
R. Willett and G. Yu, \emph{Higher index theory for certain expanders and Gromov monster groups II}, Adv. Math. 229, 2012
\bibitem[Yu]{Yu}
G. Yu, \emph{The coarse Baum--Connes conjecture for spaces which admit a uniform embedding into Hilbert space}, Inventiones Mathematicae 139, 2000

\end{thebibliography}
\end{document}